\documentclass[12pt,onecolumn,draftcls,journal,letterpaper]{IEEEtran}

\usepackage{amsfonts,amssymb,amstext,verbatim,amsmath,graphicx,color,enumerate}

\usepackage[normalem]{ulem}
\usepackage{soul} 

\usepackage{tikz}

\newtheorem{theorem}{\textbf{Theorem}}[section]
\newtheorem{assumption}[theorem]{\textbf{Assumption}}
\newtheorem{lemma}[theorem]{\textbf{Lemma}}

\newtheorem{remark}[theorem]{\textbf{Remark}}
\newtheorem{proposition}[theorem]{\textbf{Proposition}}
\newtheorem{definition}[theorem]{\textbf{Definition}}
\newtheorem{corollary}[theorem]{\textbf{Corollary}}

\newcommand{\argmax}{\mathop{\operatorname{argmax}}}

\long\def\symbolfootnote[#1]#2{\begingroup%
\def\thefootnote{\fnsymbol{footnote}}\footnote[#1]{#2}\endgroup}

\IEEEoverridecommandlockouts                              

\newcommand{\EE}{\mathcal{E}}
\newcommand{\GG}{\mathcal{G}}
\newcommand{\LL}{\mathcal{L}}

\newcommand{\nbrs}{\mathcal{N}}

\newcommand{\until}[1]{\{1,\ldots,#1\}} 
\newcommand{\subj}{\text{subj. to}}
\newcommand\oprocendsymbol{\hbox{$\square$}}
\newcommand\oprocend{\relax\ifmmode\else\unskip\hfill\fi\oprocendsymbol}
\def\eqoprocend{\tag*{$\square$}}
\newcommand{\real}{{\mathbb{R}}}
\newcommand{\1}{\mathbf{1}}
\newcommand{\map}[3]{#1: #2 \rightarrow #3}

\newcommand{\sz}[1]{\mathbf{z}^{#1}}

\newcommand{\sx}[1]{\mathbf{x}^{#1}}
\newcommand{\barsx}[1]{\bar{\mathbf{x}}^{#1}}
\newcommand{\smu}[1]{\boldsymbol{\mu}^{#1}}
\newcommand{\slambda}[1]{\boldsymbol{\lambda}^{#1}}

\newcommand{\srho}[1]{\rho^{#1}}

\newcommand{\slotIndex}{s}
\newcommand{\slotUB}{S}

\def \Dminmaxfull/{Distributed Duality-Based Peak Minimization (DDPM)}
\def \Dminmax/{DDPM}

\renewcommand{\inf}{\operatornamewithlimits{inf\vphantom{p}}}

\renewcommand{\lim}{\operatornamewithlimits{lim\vphantom{p}}}

\usepackage{algorithm}
\usepackage[noend]{algpseudocode}
\makeatletter
\newcommand{\StatexIndent}[1][3]{%
  \setlength\@tempdima{\algorithmicindent}%
  \Statex\hskip\dimexpr#1\@tempdima\relax}
\makeatother

\renewcommand{\theenumi}{(\roman{enumi})}

\graphicspath{{figs/}}

\begin{document}

\title{
  A duality-based approach for\\distributed min-max optimization
  }

  \author{Ivano Notarnicola$^1$, Mauro Franceschelli$^2$, Giuseppe
    Notarstefano$^1$ \thanks{A preliminary short version of this paper is going
      to appear as \cite{notarnicola2016duality}. The research leading to these
      results has received funding from the European Research Council (ERC)
      under the European Union's Horizon 2020 research and innovation programme
      (grant agreement No 638992 - OPT4SMART) and from the Italian grant SIR
      ``Scientific Independence of young Researchers'', project CoNetDomeSys,
      code RBSI14OF6H, funded by the Italian Ministry of Research and Education
      (MIUR).  } \thanks{ $^1$Ivano Notarnicola and Giuseppe Notarstefano are
      with the Department of Engineering, Universit\`a del Salento, Via
      Monteroni, 73100 Lecce, Italy, \texttt{name.lastname@unisalento.it.}  }
    \thanks{$^2$Mauro Franceschelli is with the Department of Electrical and
      Electronic Engineering, University of Cagliari, Piazza D'Armi, 09123
      Cagliari, Italy, \texttt{mauro.franceschelli@diee.unica.it.} } }

\maketitle

\begin{abstract}
  In this paper we consider a distributed optimization scenario in which a set
  of processors aims at cooperatively solving a class of min-max optimization
  problems. This set-up is motivated by peak-demand minimization problems in
  smart grids. Here, the goal is to minimize the peak value over a finite
  horizon with: (i) the demand at each time instant being the sum of
  contributions from different devices, and (ii) the device states at different
  time instants being coupled through local constraints (e.g., the
  dynamics). The min-max structure and the double coupling (through the devices
  and over the time horizon) makes this problem challenging in a distributed
  set-up (e.g., existing distributed dual decomposition approaches cannot be
  applied). We propose a distributed algorithm based on the combination of
  duality methods and properties from min-max optimization. Specifically, we
  repeatedly apply duality theory and properly introduce ad-hoc slack variables
  in order to derive a series of equivalent problems. On the resulting problem
  we apply a dual subgradient method, which turns out to be a distributed
  algorithm consisting of a minimization on the original primal variables and a
  suitable dual update. We prove the convergence of the proposed algorithm in
  objective value. Moreover, we show that every limit point of the primal
  sequence is an optimal (feasible) solution. Finally, we provide numerical
  computations for a peak-demand optimization problem in a network of
  thermostatically controlled loads.
\end{abstract}

\section{Introduction}
\label{sec:intro}
Distributed optimization problems arise as building blocks of several network
problems in different areas as, e.g., control, estimation and learning.
On this regard, the addition of processing, measurement, communication and
control capability to the electric power grid is leading to ``smart grids'' in which
tasks, that were typically performed at a central level, can be more efficiently
performed by smart devices in a cooperative way.
Therefore, these complex systems represent a rich source of motivating
optimization scenarios.
An interesting example is the design of smart generators, accumulators and loads
that cooperatively execute Demand Side Management (DSM)
programs~\cite{alizadeh2012demand}. The goal is to reduce the hourly and daily
variations and peaks of electric demand by optimizing generation, storage and
consumption.
A widely adopted objective in DSM programs is Peak-to-Average Ratio (PAR),
defined as the ratio between peak-daily and average-daily power demands. PAR
minimization gives rise to a min-max optimization problem if the average daily
electric load is assumed not to be affected by the demand response strategy.

This problem has been already investigated in the literature in a noncooperative
framework. In~\cite{mohsenian2010autonomous} the
authors 
propose a game-theoretic model for PAR minimization and provide a distributed
energy-cost-based strategy for the users which is proven to be optimal.
A noncooperative-game approach is also proposed in~\cite{atzeni2013demand},
where optimal strategies are characterized and a distributed scheme is designed
based on a proximal decomposition algorithm. 
It is worth pointing out that in the literature above the term ``distributed'' is
used to indicate that data are deployed on a set of devices, which
perform local computation simultaneously. However, the nodes do not run a
``distributed algorithm'', that is they do not cooperate and do not exchange
information locally over a communication graph.

Motivated by this application scenario, in this paper we propose a novel
distributed optimization framework for min-max optimization problems commonly
found in DSM problems.
Differently from the references above, we consider a cooperative, distributed
computation model in which the agents in the network solve the optimization
problem (i) without any knowledge of aggregate quantities, (ii) by communicating
only with neighboring agents, and (iii) by performing local computations (with
no central coordinator).

The distributed algorithm proposed in the paper heavily relies on duality theory.
Duality is a widely used tool for parallel and (classical) distributed
optimization algorithms as shown, e.g., in the
tutorials~\cite{palomar2006tutorial,yang2010distributed}.
More recently, in \cite{zhu2012distributed} a distributed, consensus-based,
primal-dual algorithm is proposed to solve constrained optimization problems
with separable convex costs and common convex constraints.
In~\cite{chang2014distributed} the authors use the same technique to solve
optimization problems with coupled smooth convex costs and convex inequality
constraints. In the proposed algorithm, agents employ a
consensus technique to estimate the global cost and constraint functions
and use a local primal-dual perturbed subgradient method to obtain a global optimum.
These approaches do not apply to optimization problems as the one considered in
this paper.

Primal recovery is a key issue in dual methods, since the
primal sequence is not guaranteed, in general, to satisfy the dualized primal
constraint. Thus, several strategies have been proposed to cope with this issue.
In~\cite{nedic2009approximate}, the authors propose and analyze a centralized
algorithm for generating approximate primal solutions via a dual subgradient
method applied to a convex constrained optimization problem. Moreover, in the
paper the problem of (exact) primal recovery and rate analysis of existing
techniques is widely discussed.
In \cite{lu2013convergence}, still in a centralized set-up, the primal convergence rate
of dual first-order methods is studied when the primal problem is only
approximately solved.
In \cite{simonetto2016primal} a distributed algorithm is proposed to generate
approximate dual solutions for a problem with separable cost function and
coupling constraints. 
A similar optimization set-up is considered in \cite{falsone2016dual} in a
distributed set-up. A dual decomposition approach combined with a proximal
minimization is proposed to generate a dual solution. 
In the last two papers, a primal recovery mechanism is proposed to obtain a
primal optimal solution.

%
Another tool used to develop and analyze the distributed algorithm we propose in
the paper is min-max optimization, which is strictly related to saddle-point
problems. In~\cite{nedic2009subgradient} the authors propose a subgradient
method to generate approximate saddle-points.
A min-max problem is also considered in~\cite{srivastava2011distributed} and a
distributed algorithm based on a suitable penalty approach has been
proposed.
Differently from our set-up, in \cite{srivastava2011distributed} each term of
the max-function is local and entirely known by a single agent.
Another class of algorithms exploits the exchange of active constraints among
the network nodes to solve constrained optimization problems which include
min-max problems,
\cite{notarstefano2011distributed,burger2014polyhedral}. Although they work
under asynchronous, directed communication they do not scale in set-ups as the
one in this paper, in which the terms of the max function are coupled.
Very recently, in~\cite{mateos2015distributed} the authors proposed a
distributed projected subgradient method to solve constrained saddle-point
problems with agreement constraints. The proposed algorithm is based on
saddle-point dynamics with Laplacian averaging.
Although our problem set-up fits in those considered in~\cite{mateos2015distributed},
our algorithmic approach and the analysis are different.
In~\cite{simonetto2012regularized,koppel2015regret} saddle point dynamics are
used to design distributed algorithms for standard separable optimization
problems.


%

The main contributions of this paper are as follows. First, we propose a novel
distributed optimization framework which is strongly motivated by peak
power-demand minimization in DSM. The optimization problem has a min-max
structure with local constraints at each node. Each term in the max function
represents a daily cost (so that the maximum over a given horizon needs to be
minimized), while the local constraints are due to the local dynamics and
state/input constraints of the subsystems in the smart grid.
The problem is challenging when approached in a distributed way since it is
\emph{doubly coupled}. Each term of the max function is coupled among the
agents, since it is the sum of local functions each one known by the local agent
only. Moreover, the local constraints impose a coupling between different
``days'' in the time-horizon.
The goal is to solve the problem in a distributed computation framework, in
which each agent only knows its local constraint and its local objective
function at each day.

Second, as main paper contribution, we propose a distributed algorithm to solve
this class of min-max optimization problems.
The algorithm has a very simple and clean structure in which a primal
minimization and a dual update are performed. The primal problem has a similar
structure to the centralized one.
Despite this simple structure, which resembles standard distributed dual
methods, the algorithm is not a standard decomposition scheme \cite{yang2010distributed},
and the derivation of the algorithm is non-obvious.
Specifically, the algorithm is derived by heavily resorting to duality theory
and properties of min-max optimization (or saddle-point) problems. In
particular, a sequence of equivalent problems is derived in order to decompose
the originally coupled problem into locally-coupled subproblems, and
thus being able to design a distributed algorithm.
An interesting feature of the algorithm is its expression in terms of the
original primal variables and of dual variables arising from two different
(dual) problems.  Since we apply duality more than once, and on different
problems, this property, although apparently intuitive, was not obvious a
priori.
Another appealing feature of the algorithm is that every limit point of the primal
sequence at each node is a (feasible) optimal solution of the original optimization
problem (although this is only convex and not strictly convex). This property is
obtained by the minimizing sequence of the local primal subproblems without
resorting to averaging schemes. Finally, since each
node only computes the decision variable of interest, our algorithm can solve both
large-scale (many agents are present) and big-data (a large horizon is considered)
problems.

The paper is structured as follows. In Section~\ref{sec:distributed_algorithm}
we formalize our distributed min-max optimization set-up and present the main
contribution of the paper, a novel, duality-based distributed optimization
method. In Section~\ref{sec:analysis} we characterize its convergence
properties.  In Section \ref{sec:simulations} we corroborate the theoretical
results with a numerical example involving peak power minimization in a
smart-grid scenario. Finally, in Appendix we provide some useful preliminaries
from optimization, specifically basics on duality theory and a result for the
subgradient method.


\section{Problem Set-up and Distributed\\Optimization Algorithm}
\label{sec:distributed_algorithm}
In this section we set-up the distributed min-max optimization framework and
propose a novel distributed algorithm to solve it.

\subsection{Distributed min-max optimization set-up}
\label{sec:setup}
We consider a network of $N$ processors which communicate according to a
\emph{connected}, undirected graph $\GG = (\until{N}, \EE)$, where
$\EE\subseteq \until{N} \times \until{N}$ is the set of edges. That is, the edge
$(i,j)$ models the fact that node $i$ and $j$ exchange information.  We denote
by $\nbrs_i$ the set of \emph{neighbors} of node $i$ in the fixed graph $\GG$,
i.e., $\nbrs_i := \{ j \in \until{N} \mid (i,j) \in \EE \}$. Also, we denote by
$a_{ij}$ the element $i,j$ of the adjacency matrix. We recall that $a_{ij} = 1$
if $(i,j) \in \EE$ and $i\neq j$, and $a_{ij} = 0$ otherwise.

Next, we introduce the min-max optimization problem to be solved by the network processors in a
distributed way.
Specifically, we associate to each processor $i$ a decision vector
$\sx{i} = [ \sx{i}_1, \ldots, \sx{i}_\slotUB ]^\top \in \real^\slotUB$, a
constraint set $X^i\subseteq \real^\slotUB$ and local functions
$g^i_\slotIndex$, $\slotIndex\in\until{\slotUB}$, and set-up the following
optimization problem
\begin{align}
\begin{split}
  \min_{\sx{1}, \ldots, \sx{N}} & \max_{\slotIndex \in \until{\slotUB}}
    \sum_{i=1}^N g^i_\slotIndex (\sx{i}_\slotIndex)
  \\
  \subj \: & \:\: \sx{i} \in X^i, \hspace{2cm} i\in\until{N}
\end{split}
\label{eq:minimax_starting_problem}
\end{align}
where for each $i\in\until{N}$ the set $X^i\subseteq\real^\slotUB$ is nonempty,
convex and compact, and the functions $\map{g^i_\slotIndex}{\real}{\real}$,
$\slotIndex\in\until{\slotUB}$, are convex.

Note that we use the superscript $i\in\until{N}$ to indicate that a vector
$\sx{i}\in\real^\slotUB$ belongs to node $i$, while we use the subscript $\slotIndex\in\until{\slotUB}$
to identify a vector component, i.e., $\sx{i}_\slotIndex$ is the $\slotIndex$-th
component of $\sx{i}$.




Using a standard approach for min-max problems, we introduce an auxiliary
variable $P$ to write the so called epigraph representation of
problem~\eqref{eq:minimax_starting_problem}, given by
\begin{align}
\begin{split}
  \min_{\sx{1}, \ldots, \sx{N}, P} \: & \: P
  \\
  \subj \: & \: \sx{i} \in X^i , \hspace{1.65cm} i\in\until{N}
  \\
  & \: \sum_{i=1}^N g^i_\slotIndex (\sx{i}_\slotIndex) \le P, \hspace{.5cm}  \slotIndex \in \until{\slotUB}.
\end{split}
\label{eq:starting_problem}
\end{align}

It is worth noticing that this problem has a particular structure, which gives
rise to interesting challenges in a distributed set-up. First of all, two types
of couplings are present, which involve simultaneously the $N$ agents and the
$S$ components of each decision variable $\sx{i}$. Specifically, for a given
index $s$, the constraint
$\sum_{i=1}^N g^i_\slotIndex (\sx{i}_\slotIndex) \le P$ couples all the vectors
$\sx{i}$, $i\in\until{N}$. At the same time, for a given $i\in\until{N}$, the
constraint $X^i$ couples all the components $\sx{i}_1, \ldots, \sx{i}_S$ of
$\sx{i}$. Figure~\ref{fig:cross_constraints} provides a nice graphical
representation of this interlaced coupling.  Moreover, the problem is both
large-scale and big-data. That is, both the number of decision variables and the
number of constraints depend on $N$ (and thus scale badly with the number of
agents in the network). Also, the dimension of the coupling constraint, $S$, can
be large. Therefore, common approaches as reaching a consensus among the nodes
on an optimal solution and/or exchanging constraints are not computationally
affordable.

\begin{figure}[!htbp]
  \centering
  \includegraphics[scale=0.96]{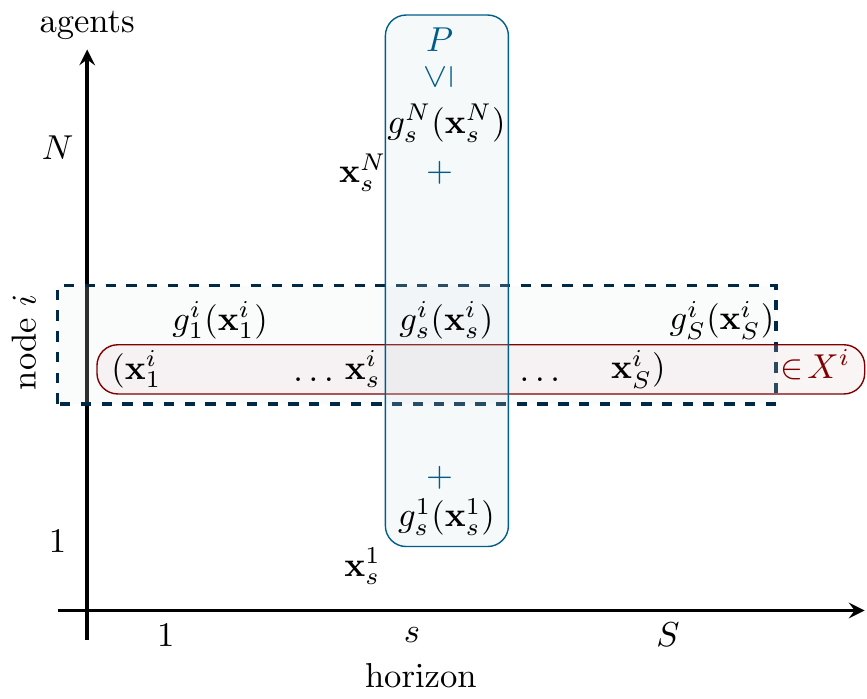}
  \caption{ Graphical representation of interlaced constraints.  }
  \label{fig:cross_constraints}
\end{figure}

To conclude this section, notice that problem~\eqref{eq:starting_problem} is
convex, but not strictly convex.
This means that it is not guaranteed to have a unique optimal
solution. As discussed in the introduction, this impacts on dual approaches when
trying to recover a primal optimal solution, see e.g.,
\cite{nedic2009approximate} and references therein.
%
This aspect is even more delicate in a distributed set-up in which nodes only
know part of the constraints and of the objective function.


\subsection{\Dminmaxfull/ }
\label{sec:alg_description}
Next, we introduce our distributed optimization algorithm.
Informally, the algorithm consists of a two-step procedure.  First, each node
$i\in\until{N}$ stores a set of variables (($\sx{i}$, $\srho{i}$), $\smu{i}$)
obtained as a primal-dual optimal solution pair of a local optimization problem
with an epigraph structure as the centralized problem. The coupling with the
other nodes in the original formulation is replaced by a term depending on
neighboring variables $\slambda{ij}$, $j\in\nbrs_i$. These variables are updated
in the second step according to a suitable linear law weighting the difference
of neighboring $\smu{i}$.
Nodes use a diminishing step-size denoted by $\gamma(t)$ and can initialize the
variables $\slambda{ij}$, $j\in\nbrs_i$ to arbitrary values.
In the next table we formally state our \Dminmaxfull/ algorithm from the
perspective of node $i$.
\begin{algorithm}
\renewcommand{\thealgorithm}{}
\floatname{algorithm}{Distributed Algorithm}

  \begin{algorithmic}[0]
    \Statex \textbf{Processor states}: $(\sx{i}, \srho{i})$, $\smu{i}$ and $\slambda{ij}$ for $j\in\nbrs_i$

    \Statex \textbf{Evolution}:

      \StatexIndent[0.5] \textbf{Gather} $ \slambda{ji}(t)$ from $j\in\nbrs_i$

      \StatexIndent[0.5] \textbf{Compute} $\big((\sx{i}(t+1),\srho{i}(t+1)),\smu{i}(t+1)\big)$ as a
        primal-dual optimal solution pair of
      \begin{align}
      \begin{split}
        \min_{\sx{i}, \srho{i}} \: & \: \srho{i}
        \\
        \subj \: & \: \sx{i} \in X^i
        \\
        & \: g^i_\slotIndex (\sx{i}_\slotIndex ) + \sum_{j\in\nbrs_i}  \big( \slambda{ij}(t) - \slambda{ji}(t) \big)_{\slotIndex} \le \srho{i},
        \\
        & \hspace{4.0cm}\slotIndex\in\until{\slotUB}
      \end{split}
      \label{eq:alg_minimization}
      \end{align}

      \StatexIndent[0.5] \textbf{Gather} $ \smu{j}(t+1)$ from $j\in\nbrs_i$

      \StatexIndent[0.5] \textbf{Update} for all $j\in\nbrs_i$
      \begin{align}
        \slambda{ij} (t \!+\! 1) = \slambda{ij}(t) - \gamma(t) (\smu{i}(t \!+\! 1) \!-\! \smu{j}(t \!+\! 1))
      \label{eq:alg_update}
      \end{align}

  \end{algorithmic}
  \caption{\Dminmax/}
  \label{alg:distributed_DSM}
\end{algorithm}

The structure of the algorithm and the meaning of the updates will be clear in
the constructive analysis carried out in the next section.
At this point we want to point out that although
problem~\eqref{eq:alg_minimization} has the same epigraph structure of
problem~\eqref{eq:starting_problem}, $\rho^i$ is not a copy of the centralized
cost $P$, but rather a local contribution to that cost. That is, as we will see,
the total cost $P$ will be the sum of the $\rho^i$s.


\section{Algorithm Analysis}
\label{sec:analysis}
The analysis of the proposed \Dminmax/ distributed algorithm is constructive
and heavily relies on duality theory tools.

We start by deriving the equivalent dual problem of~\eqref{eq:starting_problem}
which is formally stated in the next lemma.


\begin{lemma}
	The optimization problem
	\begin{align}
	\begin{split}
	  \max_{ \smu{} \in\real^\slotUB } \: & \: \sum_{i=1}^N q^i ( \smu{} )
	  \\
	  \subj \: & \: \1^\top \smu{} = 1, \: \smu{} \succeq 0,
	\end{split}
	\label{eq:dual_centr}
	\end{align}
  with $\1 := [1, \ldots, 1]^\top \in \real^\slotUB$ and
	\begin{align}
	  q^i( \smu{} ) & := \min_{\sx{i} \in X^i} \sum_{\slotIndex=1}^\slotUB
                          \smu{}_\slotIndex g^i_\slotIndex (\sx{i}_\slotIndex), 
	\label{eq:qi_definition}
	\end{align}
	for all $i\in\until{N}$, is the dual of
        problem~\eqref{eq:starting_problem}.

  Moreover, both problems~\eqref{eq:starting_problem} and \eqref{eq:dual_centr}
  have finite optimal cost, respectively $P^\star$ and $q^\star$, and strong
  duality holds, i.e.,
  \begin{align*}
    P^\star = q^\star.
  \end{align*}%
  \label{lem:equivalence_dual_and_initial} %
\end{lemma}
\begin{proof}
  We start showing that problem~\eqref{eq:dual_centr} is the dual of
  \eqref{eq:starting_problem}. Let
  $\smu{} := [\smu{}_1,\ldots,\smu{}_\slotUB]^\top\in\real^\slotUB$ be $\slotUB$ Lagrange multipliers
  associated to the inequality constraints
  $\sum_{i=1}^N g^i_\slotIndex (\sx{i}_\slotIndex) - P \le 0$ for $\slotIndex\in\until{\slotUB}$ in
  \eqref{eq:starting_problem}. Then the partial Lagrangian\footnote{We have a
  ``partial Lagrangian'' since we do not dualize all the constraints. Here local
  constraints $\sx{i} \in X^i$, $i\in\until{N}$, are not dualized.} of
  problem~\eqref{eq:starting_problem} is given by
  \begin{align*}
  \begin{split}
    \LL_1(\sx{1}, \ldots, \sx{N}, P, \smu{} )
    & = P + \sum_{\slotIndex=1}^\slotUB \smu{}_\slotIndex \Big(\sum_{i=1}^N g^i_\slotIndex (\sx{i}_\slotIndex)  - P \Big)
    \\
    & = P \Big(1-\sum_{\slotIndex=1}^\slotUB \smu{}_\slotIndex \Big) +
    \sum_{i=1}^N \sum_{\slotIndex=1}^\slotUB \smu{}_\slotIndex g^i_\slotIndex ( \sx{i}_\slotIndex ).
  \end{split}
  \end{align*}

By definition, the dual function is defined as
\begin{align*}
  q( \smu{} ) :=
    \min_{ \sx{1} \in X_1,\ldots, \sx{N} \in X_N, P }
      \LL_1 ( \sx{1}, \ldots, \sx{N}, P, \smu{} ),
\end{align*}
where the presence of constraints $\sx{i} \in X^i$ for all $i\in\until{N}$ is
due to the fact that we have not dualized them.

The minimization of $\LL_1$ with respect to $P$ gives rise to the simplex
constraint
$\sum_{\slotIndex=1}^{\slotUB} \smu{}_\slotIndex = 1$.
The minimization with respect to $\sx{i}$ splits over $i\in\until{N}$, so that
the dual function can be written as the sum of terms $q^i$ given in~\eqref{eq:qi_definition}.

To prove strong duality, we show that the strong duality theorem for convex
inequality constraints, \cite[Proposition~5.3.1]{bertsekas1999nonlinear},
applies. Since the sets $X^i$, $i\in\until{N}$, are convex (and compact), we
need to show that the inequality constraints
$\sum_{i=1}^N g^i_\slotIndex (\sx{i}_\slotIndex) - P \le 0$ for all
$\slotIndex\in\until{\slotUB}$ are convex and that there exist
$\barsx{1} \in X_1$, $\ldots$, $\barsx{N} \in X_N$ and $\bar{P}$ such that the
strict inequality holds.
Since each $g^i_\slotIndex$ and $-P$ are convex functions, then for all $\slotIndex$
each function
\begin{align*}
  \bar{g}_s ( \sx{1}_\slotIndex,\ldots,\sx{N}_\slotIndex,P) :=
  \sum_{i=1}^N g^i_\slotIndex ( \sx{i}_\slotIndex ) - P
\end{align*}
is convex. Also, since the sets $X^i$, $i\in\until{N}$ are nonempty, there exist
$\barsx{i} \in X^i$, $i\in\until{N}$, and a sufficiently large (finite)
$\bar{P}$ such that the strict inequalities
$\bar{g}( \barsx{1}_\slotIndex,\ldots,\barsx{N}_\slotIndex, \bar P) < 0$,
$\slotIndex\in\until{\slotUB}$ are satisfied and, thus, the Slater's condition holds.
Finally, since a feasible point for the convex
problem~\eqref{eq:minimax_starting_problem} always exists, then the optimal cost
$P^\star$ is finite and so $q^\star$, thus concluding the proof.
\end{proof}

In order to make problem~\eqref{eq:dual_centr} amenable for a distributed
solution, we can rewrite it in an equivalent form.
To this end, we introduce copies of the common optimization variable $\smu{}$
and coherence constraints having the sparsity of the connected graph $\GG$, thus
obtaining
\begin{align}
\begin{split}
  \max_{ \smu{1}, \ldots, \smu{N} }
  \: & \: \sum_{i=1}^N q^i( \smu{i} )
  \\
  \subj \: & \: \1^\top \smu{i} = 1,  \: \smu{i} \succeq 0, \hspace{0.7cm} i\in\until{N}
  \\ & \:
  \smu{i} = \smu{j}, \hspace{2.1cm} (i,j) \in \EE.
\end{split}
\label{eq:problem_with_copies}
\end{align}
Notice that we have also duplicated the simplex constraint, so that it becomes
\emph{local} at each node.

To solve this problem, we can use a dual decomposition approach by designing a
dual subgradient algorithm. Notice that dual methods can be applied to
\eqref{eq:problem_with_copies} since the constraints are convex and the cost
function concave.
Also, as known in the distributed optimization literature, a dual subgradient
algorithm applied to problem~\eqref{eq:problem_with_copies} would immediately
result into a distributed algorithm if functions $q^i$ were available in closed
form.

\begin{remark}
  In standard convex optimization deriving the dual of a dual problem brings
  back to a primal formulation. However, we want to stress that in what we will
  develop in the following, problem~\eqref{eq:problem_with_copies} is dualized
  rather than problem~\eqref{eq:dual_centr}. In particular, different
  constraints are dualized, namely the coherence constraints rather than the
  simplex ones. Therefore, it is not obvious if and how this leads back to a
  primal formulation. \oprocend
\end{remark}

We start deriving the dual subgradient algorithm by dualizing only the coherence
constraints. Thus, we write the partial Lagrangian
\begin{align}
\begin{split}
  & \LL_2(\smu{1},\ldots, \smu{N}, \{ \slambda{ij} \}_{(i,j) \in \EE } )
  \\
  & \hspace{2.3cm} = \sum_{i=1}^N \Big( q^i( \smu{i} ) + \sum_{j\in\nbrs_i} {\slambda{ij}}^\top (\smu{i} - \smu{j} ) \Big)
\end{split}
\label{eq:lagrangian_definition}
\end{align}
where $\slambda{ij} \in \real^\slotUB$ for all $(i,j)\in \EE$ are Lagrange multipliers associated to
the constraints $\smu{i} - \smu{j} = 0$.

Since the communication graph $\GG$ is undirected and connected, we can exploit
the symmetry of the constraints. In fact, for each $(i,j)\in\EE$ we also have
$(j,i) \in\EE$, and, expanding all the terms in
\eqref{eq:lagrangian_definition}, for given $i$ and $j$, we always have both the
terms ${\slambda{ij}}^\top (\smu{i} - \smu{j} )$ and
${\slambda{ji}}^\top (\smu{j} - \smu{i} )$.
%
Thus, after some simple algebraic manipulations, we get
\begin{align}
\begin{split}
  & \LL_2(\smu{1}, \ldots, \smu{N},  \{ \slambda{ij} \}_{(i,j) \in \EE })
  \\
  & \hspace{2.3cm} = \sum_{i=1}^N \Big( q^i( \smu{i} ) + {\smu{i} }^\top \!\! \sum_{j\in\nbrs_i} (\slambda{ij} - \slambda{ji})  \Big),
\end{split}
\label{eq:lagrangian_rearrangement}
\end{align}
which is separable with respect to $\smu{i}$.

The dual of problem~\eqref{eq:problem_with_copies} is thus
\begin{equation}
  \min_{\{\slambda{ij}\}_{(i,j)\in\EE}}
  \eta( \{ \slambda{ij} \}_{ (i,j)\in\EE} ) =
  \sum_{i=1}^N \eta^i ( \{\slambda{ij},\slambda{ji}\}_{j\in\nbrs_i} ),
\label{eq:dual_dual}
\end{equation}
with, for all $i\in\until{N}$,
\begin{align}
  \eta^i (\{\slambda{ij},\slambda{ji}\}_{j\in\nbrs_i} ) \! :=\!\!\!
  \max_{\1^\top \smu{i} = 1,\smu{i} \succeq 0}
  \! q^i( \smu{i} ) \! + \!
  {\smu{i} }^\top\!\!\! \sum_{j\in\nbrs_i} \! (\slambda{ij} \! -\!  \slambda{ji}).
  \label{eq:eta_definition}
\end{align}

In the next lemma we characterize the properties of problem~\eqref{eq:dual_dual}.
\begin{lemma}
  Problem~\eqref{eq:dual_dual}, which is the dual of
  problem~\eqref{eq:problem_with_copies}, has a bounded optimal cost, call it
  $\eta^\star$, and strong duality holds, so that
  \begin{align}
    \eta^\star = q^\star = P^\star.
  \label{eq:eta_q_P}
  \end{align}
\label{lem:strong_duality_dualdual}
\end{lemma}
\begin{proof}
  Since problem~\eqref{eq:dual_centr} is a dual problem, its cost function
  $\sum_{i=1}^N q^i( \smu{} )$ is concave on its domain, which is convex
  (simplex constraint). Moreover, by
  Lemma~\ref{lem:equivalence_dual_and_initial} its optimal cost $q^\star$ is
  finite.  Problem~\eqref{eq:problem_with_copies} is an equivalent formulation
  of~\eqref{eq:dual_centr} and, thus, has the same (finite) optimal cost
  $q^\star$. This allows us to conclude that strong duality holds between
  problem~\eqref{eq:problem_with_copies} and its dual~\eqref{eq:dual_dual}, so
  that, $\eta^\star = q^\star$. The second equality in~\eqref{eq:eta_q_P} holds
  by Lemma~\ref{lem:equivalence_dual_and_initial}, so that the proof follows.%
\end{proof}

Problem~\eqref{eq:dual_dual} has a particularly appealing structure for
distributed computation. In fact, the cost function is separable and each term
$\eta^i$ of the cost function depends only on neighboring variables
$\slambda{ij}$ and $\slambda{ji}$ with $j\in\nbrs_i$. Thus, a subgradient method
applied to this problem turns out to be a distributed algorithm.
Since problem~\eqref{eq:dual_dual} is the dual of~\eqref{eq:problem_with_copies}
we recall, \cite[Section~6.1]{bertsekas1999nonlinear}, how to compute a
subgradient of $\eta$ with respect to each component, that is,
\begin{align}
\frac{\tilde \partial
   \eta ( \{ \slambda{ij} \}_{(i,j) \in \EE } ) }{\partial \slambda{ij}}
  = {\smu{i}}^\star - {\smu{j}}^\star,
\label{eq:eta_subgradient}
\end{align}
where $\frac{\tilde \partial \eta (\cdot)}{\partial \slambda{ij}}$ denotes the component
associated to the variable $\slambda{ij}$ of a subgradient of $\eta$, and
\begin{align*}
  {\smu{k}}^\star \in \argmax_{ \1^\top \smu{k} = 1,\smu{k} \succeq 0 } \bigg( q_k( \smu{k} ) +
  {\smu{k}}^\top \sum_{h\in\nbrs_k} (\slambda{kh} - \slambda{hk}) \bigg),
\end{align*}
for $k=i,j$.
%

The distributed dual subgradient algorithm for problem \eqref{eq:problem_with_copies} can be
summarized as follows. For each node $i\in\until{N}$:
\begin{itemize}
\item[(S1)]\label{item:s1} receive $\slambda{ji}(t)$, for each $j \in \nbrs_i$, and
compute a subgradient $\smu{i}(t+1)$ by solving
\begin{align}
  \begin{split}
    \max_{\smu{i} } \: & \: q^i( \smu{i} ) +
    {\smu{i}}^\top \!\! \sum_{j\in\nbrs_i} (\slambda{ij}(t) - \slambda{ji}(t))
    \\
    \subj \:   & \: \1^\top \smu{i} = 1,\smu{i} \succeq 0.
  \end{split}
\label{eq:dual_subgradient}
\end{align}
\item[(S2)]\label{item:s2} exchange with neighbors the updated $\smu{j}(t+1)$, $j \in
\nbrs_i$, and update $\slambda{ij}$, $j\in\nbrs_i$, via
\begin{align*}
  \slambda{ij} (t \!+\! 1) = \slambda{ij}(t) - \gamma(t) (\smu{i}(t \!+\! 1) \!-\! \smu{j}(t \!+\! 1)).
\end{align*}
where $\gamma(t)$ denotes a diminishing step-size satisfying
Assumption~\ref{ass:step-size} in Appendix~\ref{app:subgradient_method}.
\end{itemize}

It is worth noting that in~\eqref{eq:dual_subgradient} the value of
$\slambda{ij}(t)$ and $\slambda{ji}(t)$, for $j\in\nbrs_i$, is fixed as
highlighted by the index $t$.
Moreover, we want to stress, once again, that the algorithm is \emph{not}
implementable as it is written, since functions $q^i$ are not available in
closed form.

On this regard, we point out that here we slightly abuse notation since in
(S1)-(S2) we use $\smu{i}(t)$ as in the \Dminmax/ algorithm, but without proving
its equivalence yet.  Since we will prove it in the next, we preferred not to
overweight the notation.

Before proving the convergence of the updates (S1)-(S2) we need the following lemma.
\begin{lemma}
  For each $i\in\until{N}$, the function $\smu{i} \mapsto q^i(\smu{i})$
  defined in~\eqref{eq:qi_definition} is concave over $ \smu{i}\succeq 0$.
\label{lem:qi_concavity}
\end{lemma}
\begin{proof}
  For each $i\in\until{N}$, consider the (feasibility) convex problem
\begin{align*}
  \min_{\sz{i} \in X^i} \: & \: 0
  \\
  \subj \: & \:  g^i_\slotIndex (\sz{i}_\slotIndex) \leq 0, \hspace{0.3cm} \slotIndex \in \until{\slotUB}.
\end{align*}

Then, $q^i(\smu{i})$ is the dual function of that problem and, thus, 
is a concave function on its domain, namely $ \smu{i}\succeq 0$.%
\end{proof}

We can now prove the convergence in objective value of the dual subgradient.
\begin{lemma}
  The dual subgradient updates (S1)-(S2), with step-size $\gamma(t)$ satisfying
  Assumption~\ref{ass:step-size}, generate sequences
  $\{ \slambda{ij}(t) \}$, $(i,j)\in \EE$, that converge in objective value to
  the optimal cost $\eta^\star$ of problem~\eqref{eq:dual_dual}.
  \label{lem:dual_subgradient_correcteness}
\end{lemma}
\begin{proof}
  As already recalled in equation~\eqref{eq:eta_subgradient}, we can build subgradients
  of $\eta$ by solving problem in the form~\eqref{eq:dual_subgradient}.
  Since in~\eqref{eq:dual_subgradient}, the maximization of the concave
  (Lemma~\ref{lem:qi_concavity}) function $q^i$ is performed over the nonempty,
  compact (and convex) probability simplex $\1^\top \smu{i}=1$,
  $\smu{i} \succeq 0$, then the maximum is always attained at a finite value. As
  a consequence, at each iteration the subgradients of $\eta$ are bounded
  quantities.  Moreover, the step-size $\gamma(t)$ satisfies
  Assumption~\ref{ass:step-size} and, thus, we can invoke
  Proposition~\ref{prop:subgradient_convergence} which guarantees that (S1)-(S2)
  converges in objective value to the optimal cost $\eta^\star$ of
  problem~\eqref{eq:dual_dual} so that the proof follows.
\end{proof}

We can explicitly rephrase update~\eqref{eq:dual_subgradient} by plugging in the
definition of $q^i$, given in~\eqref{eq:qi_definition}, thus obtaining the
following max-min optimization problem
\begin{align}
  \max_{ \1^\top \smu{i} = 1,\smu{i} \succeq 0 } \! \bigg(\!
  \min_{\sx{i} \in X^i} \!
  \sum_{\slotIndex=1}^\slotUB \smu{i}_\slotIndex \Big(
  g^i_\slotIndex ( \sx{i}_\slotIndex) \! + \!\!
  \sum_{j\in\nbrs_i} \!\! (\slambda{ij}(t) \!-\! \slambda{ji}(t))_\slotIndex \Big) \!\!\bigg).
\label{eq:maxmin}
\end{align}
Notice that \eqref{eq:maxmin} is a local problem at each node $i$ once
$\slambda{ij}(t)$ and $\slambda{ji}(t)$ for all $j\in\nbrs_i$ are given.
Thus, the dual subgradient algorithm (S1)-(S2) could be implemented in a
distributed way by letting each node $i$ solve problem~\eqref{eq:maxmin} and
exchange $\slambda{ij}(t)$ and $\slambda{ji}(t)$ with neighbors $j\in\nbrs_i$.
Next we further explore the structure of \eqref{eq:maxmin} to prove that
\Dminmax/ solves the original problem~\eqref{eq:starting_problem}.

The next lemma is a first instrumental result.
\begin{lemma}
Consider the optimization problem
\begin{align}
\begin{split}
  \max_{\smu{i} } \: & \:
     \sum_{\slotIndex=1}^\slotUB \smu{i}_\slotIndex \Big(
        g^i_\slotIndex (\sx{i}_\slotIndex) + \sum_{j\in\nbrs_i} (\slambda{ij}(t) - \slambda{ji}(t) )_{\slotIndex}
      \Big)
  \\
  \subj \: & \: \1^\top \smu{i} = 1, \smu{i} \succeq 0,
\end{split}
\label{eq:inner_maximization_problem}
\end{align}
with given $\sx{i}$, $\slambda{ij}(t)$ and $\slambda{ji}(t)$, $j\in\nbrs_i$.
Then, the problem
\begin{align}
\begin{split}
  \min_{\srho{i}} \: & \: \srho{i}
  \\
  \subj \:
   & \: g^i_\slotIndex (\sx{i}_\slotIndex ) +\sum_{j\in\nbrs_i} (\slambda{ij}(t) - \slambda{ji}(t) )_{\slotIndex} \le \srho{i},
  \\
  & \hspace{4.5cm} \slotIndex\in\until{\slotUB}
\end{split}
\label{eq:rho_formulation}
\end{align}
is dual of \eqref{eq:inner_maximization_problem} and strong duality holds.
\label{lem:inner_maximization_dual}
\end{lemma}
\begin{proof}
First, since $\sx{i}$ (as well as $\slambda{ij}(t)$ and $\slambda{ji}(t)$) is given, problem~\eqref{eq:inner_maximization_problem}
is a feasible \emph{linear} program (the simplex constraint is nonempty) and, thus, strong duality holds.
Introducing a scalar multiplier $\srho{i}$ associated to the constraint $\1^\top \smu{i} = 1$, we
write the partial Lagrangian of~\eqref{eq:inner_maximization_problem}
\begin{align*}
  \LL_3(\smu{i} , \srho{i}) & = \sum_{\slotIndex=1}^\slotUB \smu{i}_\slotIndex \Big( g^i_\slotIndex (\sx{i}_\slotIndex)
    + \sum_{j\in\nbrs_i} (\slambda{ij}(t) - \slambda{ji}(t) )_{\slotIndex} \Big)
    \\
    & \hspace{4.5cm} + \srho{i} ( 1 - \1^\top \smu{i} )
\end{align*}
and rearrange it as
\begin{align*}
  \LL_3(\smu{i} , \srho{i}) \! =\! \sum_{\slotIndex=1}^\slotUB \! \smu{i}_\slotIndex \Big( g^i_\slotIndex (\sx{i}_\slotIndex)
    \!+ \!\! \! \sum_{j\in\nbrs_i} ( \slambda{ij}(t) \!-\! \slambda{ji}(t) )_{\slotIndex} \!-\! \srho{i}\Big) \!+\! \srho{i}\!.
\end{align*}

The dual function $\max_{ \smu{i} \succeq 0 } \LL_3(\smu{i} , \srho{i})$ is
equal to $\srho{i}$ with domain given by the inequalities
$\srho{i} \geq g^i_\slotIndex (\sx{i}_\slotIndex ) +\sum_{j\in\nbrs_i} (\slambda{ij}(t) -
\slambda{ji}(t) )_{\slotIndex}$, $\slotIndex\in\until{\slotUB}$.
Thus, the dual problem is obtained by maximizing the dual function over its
domain giving~\eqref{eq:rho_formulation}, so that the proof follows.
\end{proof}

The next lemma is a second instrumental result.

\begin{lemma}
  Max-min optimization problem~\eqref{eq:maxmin} is the saddle point problem
  associated to problem~\eqref{eq:alg_minimization}
  \begin{align*}
    \min_{\sx{i}, \srho{i}} \: & \: \srho{i}
    \\
    \subj \: & \: \sx{i} \in X^i
    \\
    & \: g^i_\slotIndex (\sx{i}_\slotIndex ) + \sum_{j\in\nbrs_i}  \big( \slambda{ij}(t) - \slambda{ji}(t) \big)_{\slotIndex} \le \srho{i},
    \\
    & \hspace{4.3cm}\slotIndex\in\until{\slotUB}.
  \end{align*}
  Moreover, a primal-dual optimal solution pair of~\eqref{eq:alg_minimization},
  call it $( ( \sx{i} (t+1), \srho{i} (t+1) ), \smu{i}(t+1) )$, exists and
  $( \sx{i} (t+1), \smu{i}(t+1) )$ is a solution of~\eqref{eq:maxmin}.
\label{lem:dual_minmax_equivalence}
\end{lemma}%
\begin{proof}
We give a constructive proof which clarifies how the problem~\eqref{eq:alg_minimization}
is derived from~\eqref{eq:maxmin}.

Define
\begin{align}
\phi(\sx{i},\smu{i}):=\sum_{\slotIndex=1}^\slotUB \smu{i}_\slotIndex \Big(
  g^i_\slotIndex ( \sx{i}_\slotIndex) \! + \! \sum_{j\in\nbrs_i} (\slambda{ij}(t) - \slambda{ji}(t))_\slotIndex \Big)
\end{align}
and note that (i) $\phi(\cdot,\smu{i})$ is closed and convex for all
$\smu{i} \succeq 0$ (affine transformation of a convex function with compact
domain $X^i$) and (ii) $\phi(\sx{i}, \cdot )$ is closed and concave since it is
a linear function with compact domain ($\1^\top \smu{i} = 1$, $\smu{i} \succeq 0$), for
all $\sx{i}\in\real^\slotUB$.
Thus we can invoke Proposition~\ref{prop:saddle_point} which allows us
to switch $\max$ and $\min$ operators, and write
\begin{align}
\begin{split}
  & \max_{ \1^\top \smu{i} = 1,\smu{i} \succeq 0 } \! \bigg(
  \! \min_{\sx{i} \in X^i}
  \sum_{\slotIndex=1}^\slotUB \smu{i}_\slotIndex \Big(
  g^i_\slotIndex ( \sx{i}_\slotIndex) \! + \!\! \sum_{j\in\nbrs_i} \! (\slambda{ij} (t) \!-\! \slambda{ji} (t) )_\slotIndex \! \Big) \! \!\! \bigg)
  \\
  & = \! \!
  \min_{\sx{i} \in X^i} \!\! \bigg( \!
  \max_{ \1^\top \smu{i} = 1,\smu{i} \succeq 0 } \!
  \sum_{\slotIndex=1}^\slotUB \smu{i}_\slotIndex \! \Big( 
  g^i_\slotIndex (\sx{i}_\slotIndex) \! + \!\! \! \sum_{j\in\nbrs_i} \! (\slambda{ij}(t) \!-\!  \slambda{ji} (t) )_\slotIndex \! \Big) \! \! \! \bigg).
\end{split}
\label{eq:minmax}
\end{align}

Since the inner maximization problem depends nonlinearly on $\sx{i}$
(which is itself an optimization variable), it cannot be performed without
also considering the simultaneous minimization over $\sx{i}$. We overcome this issue
by substituting the inner maximization problem with its equivalent dual minimization.
In fact, by Lemma~\ref{lem:inner_maximization_dual} we can rephrase the right
hand side of~\eqref{eq:minmax} as
\begin{align}
  \min_{ \sx{i} \in X^i } \!\! \bigg(
  \min_{
    \stackrel{ \scriptstyle \srho{i}\: : \: g^i_\slotIndex (\sx{i}_\slotIndex) +\sum_{j\in\nbrs_i}
    (\slambda{ij}(t) - \slambda{ji}(t) )_{\slotIndex} \le \srho{i}}{\slotIndex\in\until{\slotUB} }
  } \: \: \srho{i}
\bigg).
\label{eq:min_min_rho}
\end{align}
At this point, a \emph{joint} (constrained) minimization with respect to $\sx{i}$ and $\srho{i}$ can be
simultaneously performed leading to problem~\eqref{eq:alg_minimization}.

To prove the second part, namely that a primal-dual optimal solution pair exists
and solves problem~\eqref{eq:maxmin}, we first notice that
problem~\eqref{eq:alg_minimization} is convex. Indeed, the cost function is
linear and the constraints are convex ($X^i$ is convex as well as the functions
$g^i_\slotIndex (\sx{i}_\slotIndex ) + \sum_{j\in\nbrs_i} \big( \slambda{ij}(t)
- \slambda{ji}(t) \big)_{\slotIndex}$ and $-\srho{i}$).
Then, by using similar arguments as in
Lemma~\ref{lem:equivalence_dual_and_initial}, we can show that the problem
satisfies the Slater's constraint qualification and, thus, strong duality
holds. Therefore, a primal-dual optimal solution pair
$(\sx{i}(t+1), \srho{i}(t+1),\smu{i}(t+1))$ exists and from the previous
arguments $(\sx{i}(t+1),\smu{i}(t+1))$ solves~\eqref{eq:maxmin}, thus concluding
the proof.
\end{proof}

\begin{remark}[Alternative proof of Lemma~\ref{lem:dual_minmax_equivalence}]
  Let $\smu{i}_\slotIndex \ge 0$, $\slotIndex\in\until{\slotUB}$ be (nonnegative)
  Lagrange multipliers associated to the inequality constraints of problem~\eqref{eq:alg_minimization}.
  Then, its (partial) Lagrangian can be written as
  \begin{align*}
     & \LL_4(\srho{i},\sx{i},\smu{i}) =
      \srho{i} + \sum_{\slotIndex = 1}^\slotUB {\smu{i}_\slotIndex} \Big( g^i_\slotIndex (\sx{i}_\slotIndex )
       \\ &\hspace{3.5 cm}
       + \sum_{j\in\nbrs_i} \big( \slambda{ij}(t) - \slambda{ji}(t) \big)_{\slotIndex} - \srho{i} \Big)
  \end{align*}
  and collecting the multiplier $\srho{i}$, we obtain
  \begin{align*}
    \LL_4(\srho{i},\sx{i},\smu{i}) & =
    \srho{i} (1-\sum_{\slotIndex = 1}^\slotUB {\smu{i}_\slotIndex})
    \\
    &\hspace{0.2cm}
    + \! \sum_{\slotIndex = 1}^\slotUB {\smu{i}_\slotIndex} \Big( g^i_\slotIndex (\sx{i}_\slotIndex )
    + \! \! \sum_{j\in\nbrs_i}  \big( \slambda{ij}(t) \!-\! \slambda{ji}(t) \big)_{\slotIndex}\Big).
  \end{align*}
  The minimization of $\LL_4$ with respect to $\srho{i}$ constrains the $1$-norm
  of the dual variable $\smu{i}$ (i.e., $\1^\top\smu{i}=1$). Then, minimizing
  the reminder over $\sx{i} \in X^i$ and maximizing the result over $\smu{i} \succeq 0$
  gives problem~\eqref{eq:maxmin}.~\oprocend
\end{remark}

We point out that in the previous lemma we have shown that the minimization
in~\eqref{eq:alg_minimization} turns out to be equivalent to performing
step~(S1).  
An important consequence of Lemma~\ref{lem:inner_maximization_dual} is that each
iteration of the algorithm can be in fact performed (since a prima-dual optimal
solution pair of \eqref{eq:alg_minimization} exists). This is strictly related
to the result of Lemma~\ref{lem:dual_subgradient_correcteness}. In fact, the
solvability of problem~\eqref{eq:alg_minimization} is equivalent to the
boundedness, at each $t$, of the subgradients of $\eta$. This is ensured,
equivalently, by the compactness of the simplex constraint in
\eqref{eq:dual_subgradient}.

The next corollary is a byproduct of the proof of Lemma~\ref{lem:dual_minmax_equivalence}.
\begin{corollary}
  Let $\srho{i}(t+1)$, for each $i\in\until{N}$, be the optimal cost of
  problem~\eqref{eq:alg_minimization} with fixed values
  $\{\slambda{ij} (t),\slambda{ji} (t) \}_{j\in\nbrs_i}$.  Then, it holds that
  \begin{align}
    \srho{i}(t+1) = \eta^i (\{\slambda{ij} (t) ,\slambda{ji} (t) \}_{j\in\nbrs_i} )
    \label{eq:rho_eta}
  \end{align}
  where $\eta^i$ is defined in~\eqref{eq:eta_definition}.
  \label{cor:eta_rho}
\end{corollary}
\begin{proof}
  To prove the corollary, 
  we first rewrite explicitly the definition of
  $\eta^i (\{\slambda{ij} (t) ,\slambda{ji} (t) \}_{j\in\nbrs_i} )$ given
  in~\eqref{eq:eta_definition}, i.e.,
\begin{align}
\begin{split}
  & \eta^i (\{\slambda{ij} (t) ,\slambda{ji} (t) \}_{j\in\nbrs_i} ) =
  \\
  & \hspace{1.8cm}
  \max_{ \1^\top \smu{i} = 1,\smu{i} \succeq 0 } \! \bigg(
    \! \Big( \min_{\sx{i} \in X^i}
      \sum_{\slotIndex=1}^\slotUB \smu{i}_\slotIndex 
      g^i_\slotIndex ( \sx{i}_\slotIndex) \! \Big)
  \\ 
  & \hspace{3.75cm}
  + \! {\smu{i}}^\top \! \! \sum_{j\in\nbrs_i} \!
  (\slambda{ij} (t) \!-\! \slambda{ji} (t) ) \! \bigg).
\end{split}
\label{eq:eta_explicit}
\end{align}
Then, being $\rho_i(t)$ the optimal cost of problem~\eqref{eq:alg_minimization},
it is also the optimal cost of problem~\eqref{eq:min_min_rho}, which is
equivalent to the right hand side of equation~\eqref{eq:minmax}.
The proof follows by noting that the expression of $\eta^i$ in
\eqref{eq:eta_explicit} is exactly the left hand side of \eqref{eq:minmax} after
rearranging some terms.
\end{proof}

We are now ready to state the main result of the paper, namely the convergence
of the \Dminmax/ distributed algorithm.

\begin{theorem}
  Let $\{ (\sx{i}(t), \srho{i}(t)) \}$,
  $i\in\until{N}$, 
  be a sequence generated by the \Dminmax/ distributed algorithm, with
  $\gamma(t)$ satisfying Assumption~\ref{ass:step-size}.
  Then, the following holds:
  \begin{enumerate}
  \item the sequence $\big\{ \sum_{i=1}^N \rho^i(t) \big\}$ converges to the
    optimal cost $P^\star$ of problem~\eqref{eq:minimax_starting_problem}, and
  \item every limit point of the primal sequence $\{ \sx{i}(t) \}$, with
    $i\in\until{N}$, is an optimal (feasible) solution
    of~\eqref{eq:minimax_starting_problem}.
  \end{enumerate}
  \label{thm:convergence}
\end{theorem}
\begin{proof}
  We prove the theorem by combining all the results given in the previous
  lemmas.

  First, for each $i\in\until{N}$, let $\{\smu{i}(t) \}$, and
  $\{\slambda{ij}(t)\}$, $j\in \nbrs_i$, be the auxiliary sequences defined in
  the \Dminmax/ distributed algorithm associated to $\{ (\sx{i}(t), \srho{i}(t)) \}$.
  From Lemma~\ref{lem:dual_minmax_equivalence} a primal-dual optimal solution
  pair $( ( \sx{i} (t+1), \srho{i} (t+1) ), \smu{i}(t+1) )$
  of~\eqref{eq:alg_minimization} in fact exists (so that the algorithm is
  well-posed) and $( \sx{i} (t+1), \smu{i}(t+1) )$
  solves~\eqref{eq:maxmin}. Recalling that solving~\eqref{eq:maxmin} is
  equivalent to solving \eqref{eq:dual_subgradient}, it follows that
  $\smu{i}(t+1)$ in the \Dminmax/ implements step (S1) of
  the dual subgradient (S1)-(S2).
  Noting that update~\eqref{eq:alg_update} of $\slambda{ij}$ is exactly step
  (S2), it follows that \Dminmax/ is an operative way to implement the dual
  subgradient algorithm (S1)-(S2).
  From Lemma~\ref{lem:dual_subgradient_correcteness} the algorithm converges in
  objective value, that is
\begin{align*}
\lim_{t\to\infty}\sum_{i=1}^N \eta^i (\{\slambda{ij} (t) ,\slambda{ji} (t)
  \}_{j\in\nbrs_i} )=\eta^\star=P^\star,
\end{align*}
where the second equality follows from Lemma~\ref{lem:strong_duality_dualdual}.
Then, we notice that from Corollary~\ref{cor:eta_rho}
\begin{align*}
\sum_{i=1}^N \rho^i(t) = \sum_{i=1}^N \eta^i (\{\slambda{ij} (t)
  ,\slambda{ji} (t) \}_{j\in\nbrs_i} ) \qquad \forall\, t\geq0,
\end{align*}
so that $\lim_{t\to\infty}\sum_{i=1}^N \rho^i(t) = P^\star$,
thus concluding the proof of the first statement.

To prove the second statement, we show that every limit point of the (primal)
sequence $\{ \sx{i}(t) \}$, $i\in\until{N}$, is feasible and optimal for
problem~\eqref{eq:minimax_starting_problem}.

For analysis purposes, let us introduce the sequence $\{P(t)\}$ defined as
  \begin{align}
    P(t) := \max_{\slotIndex \in\until{\slotUB}} \sum_{i=1}^N g^i_\slotIndex (\sx{i}_\slotIndex(t) )
    \label{eq:Pt_definition}
  \end{align}
  for each $t\ge 0$. 
  Notice that $P(t)$ is also the cost of problem~\eqref{eq:starting_problem}
  associated to $[ \sx{1} (t),\ldots, \sx{N} (t), P(t) ]$ and thus, by
  definition of optimality, satisfies
  \begin{align}
    P^\star \le P(t)
    \label{eq:cost_minorization}
  \end{align}
   for all $t\ge 0$.

  By summing over $i\in\until{N}$ both sides of inequality constraints
  in~\eqref{eq:alg_minimization}, at each $t\geq0$ the following holds
  \begin{align}
    \sum_{i=1}^N \Big( g^i_\slotIndex (\sx{i}_\slotIndex(t) ) -\!\! \sum_{j\in\nbrs_i} (\slambda{ij}(t) - \slambda{ji}(t) )_{\slotIndex} \Big)
    \le \! \sum_{i=1}^N \srho{i}(t).
    \label{eq:proof_rho_property}
  \end{align}
  Let us denote $a_{ij}$ the $(i,j)$-th entry of the adjacency matrix associated
  to the undirected graph $\GG$. Then, we can write
  \begin{align*}
    \sum_{i=1}^N \sum_{j\in\nbrs_i} &(\slambda{ij}(t) - \slambda{ji}(t) )\\
   & =  \sum_{i=1}^N \sum_{j=1}^N a_{ij}(\slambda{ij}(t) - \slambda{ji}(t) ) \\
    &=\sum_{i=1}^N \sum_{j=1}^N a_{ij}\slambda{ij}(t) - \sum_{i=1}^N \sum_{j=1}^N a_{ij}\slambda{ji}(t).
  \end{align*}
Since the graph $\GG$ is undirected $a_{ij} = a_{ji}$ for all $(i,j)\in\EE$ and thus
  \begin{align*}
\sum_{i=1}^N \sum_{j\in\nbrs_i} &(\slambda{ij}(t) - \slambda{ji}(t) )\\
   & =\sum_{i=1}^N \sum_{j=1}^N a_{ij}\slambda{ij}(t) - \sum_{i=1}^N
     \sum_{j=1}^N a_{ji}\slambda{ji}(t) = 0.
  \end{align*}
Hence,~\eqref{eq:proof_rho_property} reduces to
  \begin{align}
    \sum_{i=1}^N g^i_\slotIndex (\sx{i}_\slotIndex(t) )
    \le \sum_{i=1}^N \srho{i}(t),
  \label{eq:feasibility_condition}
  \end{align}
  for all $\slotIndex \in\until{\slotUB}$ and $t\ge 0$.

  For all $i\in\until{N}$, since $\{ \sx{i}(t)\}$ is a bounded sequence in $X^i$, then there
  exists a convergent sub-sequence $\{ \sx{i}(t_n) \}$. Let $\barsx{i}$ be its limit point.
  Since each $g^i_\slotIndex$ is a (finite) convex function over $\real$, it is
  also continuous over any compact subset of $\real$
  and, taking the limit of~\eqref{eq:feasibility_condition}, we can write
  \begin{align}
  \begin{split}
    \lim_{n\to\infty} \sum_{i=1}^N g^i_\slotIndex ( \sx{i}_\slotIndex(t_n) )
    & =\sum_{i=1}^N g^i_\slotIndex \Big( \lim_{n\to\infty} \sx{i}_\slotIndex(t_n) \Big)
    \\ &
    = \sum_{i=1}^N g^i_\slotIndex ( \barsx{i}_\slotIndex )
    \le \lim_{n\to\infty} \sum_{i=1}^N \srho{i}(t_n) = P^\star
  \end{split}
  \label{eq:feasibility_limit}
  \end{align}
  for $\slotIndex \in\until{\slotUB}$, where the last equality follows from the
  first statement of the theorem.
  Since the sub-sequence $\{ \sx{i}(t_n) \}$ is arbitrary, we have shown that every limit point
  $\barsx{i}$, $i \in\until{N}$, is feasible.

  To show optimality, first notice that in light of
  conditions~\eqref{eq:cost_minorization} and~\eqref{eq:feasibility_condition}
  the following holds
  \begin{align}
    P^\star \le P(t) = \max_{\slotIndex \in\until{\slotUB}} \sum_{i=1}^N g^i_\slotIndex (\sx{i}_\slotIndex(t) )
    \le \sum_{i=1}^N \srho{i}(t).
    \label{eq:carabinieri}
  \end{align}
  Therefore, taking any convergent sub-sequence $\{ \sx{i}(t_n) \}$ (with limit
  point $\barsx{i}$) in \eqref{eq:carabinieri}, the limit as $n\to\infty$ satisfies
  \begin{align}
    P^\star \!\! \le \!\!\lim_{n\to\infty} \!\! \bigg( \! \max_{\slotIndex \in \until{\slotUB}}
    \sum_{i=1}^N g^i_\slotIndex (\sx{i}_\slotIndex (t_n) ) \! \bigg)
    \!\! \le\!\! \lim_{n\to\infty} \sum_{i=1}^N \srho{i}(t_n) = P^\star.
  \end{align}
  By noting that the maximization is over a finite set and recalling that $g^i_\slotIndex$
  is continuous over any compact subset of $\real$, it follows
  \begin{align}
    P^\star \le \max_{\slotIndex \in \until{\slotUB}}\sum_{i=1}^N
    g^i_\slotIndex (\barsx{i}_\slotIndex ) \leq P^\star
  \label{eq:cost_bounds}
  \end{align}
  proving that any limit point $\barsx{i}$, $i \in\until{N}$, is also optimal,
  thus concluding the proof.
\end{proof}


\section{Numerical Simulations}
\label{sec:simulations}

In this section we propose a numerical example in which we apply the proposed
method to a network of Thermostatically Controlled Loads (TCLs) (such as air conditioners, 
heat pumps, electric water heaters),~\cite{Alizadeh2015reduced}.

The dynamical model of the $i$-th device is given by
\begin{align}
  \dot{T}^{i}(\tau)= -\alpha \left(T^{i}(\tau)-T^{i}_{out}(\tau)\right)+\delta^{i}(\tau)+ Qx^{i}(\tau),
  \label{eq:agent_model}
\end{align}
where $T_i(\tau)\geq0$ is the temperature, $\alpha > 0$ is a parameter depending
on geometric and thermal characteristics, $T^{i}_{out}(\tau)$ is the air
temperature outside the device, $\delta^{i}(\tau)$ represents a known
time-varying forcing term onto the internal temperature of the device,
$x^{i}(\tau)\in [0,1 ]$ is the control input, and $Q>0$ is a scaling factor.

We consider a discretized version of the system with constant input over the
sampling interval $\Delta\tau$, i.e., $x^{i}(\tau)=\sx{i}_{\slotIndex}$ for
$\tau\in \left[\slotIndex \Delta\tau, (\slotIndex+1) \Delta\tau \right)$, and
sampled state $T^{i}_{\slotIndex}$,
\begin{align}
  T^{i}_{\slotIndex+1} =
  T^{i}_{\slotIndex} e^{-\alpha \Delta\tau} \!+\! \left(1 - e^{-\alpha \Delta\tau}\right)\!\!
  \left(
    \frac{Q}{\alpha} \sx{i}_{\slotIndex} \!+\! \frac{\delta^i_{\slotIndex}}{\alpha}
    \!+\! T^{i}_{out,\slotIndex}
  \! \right)\!.
  \label{eq:agent_model_discrete_time}
\end{align}
Moreover, we constrain the temperature to stay within a given interval
$[ T_{min}, T_{max} ]$.

The constraints due to the dynamics and the bound on the temperature can be
written as inequality constraints on the input in the form $A_i\sx{i}\leq b_i$,
for each agent $i\in\until{N}$.
To construct $A_i$ and $b_i$, let us denote $\hat{A}=e^{-\alpha \Delta \tau}$ and
$\hat{B}=1-e^{-\alpha \Delta \tau}$. 
We can compute the trajectory of $T^{i}_\slotIndex$ as
a function of $\sx{i}_{\slotIndex}$, $\delta^i_\slotIndex$, $T^{i}_{out,\slotIndex}$ and $T^{i}_0$ as follows. 
Let $\delta^i, \bar{T}^{i}, \bar{T}^{i}_{out}\in \real^\slotUB$ be vectors whose $\slotIndex$-th element
corresponds respectively to $\delta^i_\slotIndex$, $T^{i}_\slotIndex$ and $T^{i}_{out,\slotIndex}$. Then, based
on~\eqref{eq:agent_model_discrete_time} it holds
\renewcommand*{\arraycolsep}{1pt}
\begin{align*}
  \bar{T}^i & =
  \underbrace{
  \begin{bmatrix}
    \hat{B} &  0 & \cdots & 0 \\
    \hat{A}\hat{B} & \hat{B} & \cdots & 0 \\
    \vdots \\
    \hat{A}^{\slotUB}\hat{B}  & \hat{A}^{\slotUB-1}\hat{B}  & \cdots & \hat{B}
  \end{bmatrix}
  }_{F} \!\!
  \Big( \!\! -\! \bar{T}^i_{out} \!+\! \frac{\delta^i}{\alpha} \!+\! \frac{Q}{\alpha} \sx{i} \Big)
  \!+\! \underbrace{
  \begin{bmatrix}
    \hat{A} \\
    \hat{A}^2 \\
    \vdots \\
    \hat{A}^{\slotUB }
  \end{bmatrix}
  }_{G} \! T^{i}_0.
\end{align*}

Thus, the matrix $A_i$ and the vector $b_i$ turn out to be
\begin{align*}
  A_i =
  \begin{bmatrix}
    \dfrac{Q}{\alpha}F \\[0.8em]
    -\dfrac{Q}{\alpha} F
  \end{bmatrix},
  \quad
  b_i =
  \begin{bmatrix}
    T_{max} \1 -G T^i_0 + F \bar{T}^i_{out} + F\dfrac{\delta^i}{\alpha} \\[0.8em]
    -T_{min} \1 +G T^i_0 -F \bar{T}^i_{out} - F\dfrac{\delta^i}{\alpha} 
  \end{bmatrix}.
\end{align*}

We assume that the power consumption $g^i_\slotIndex ( \sx{i}_\slotIndex)$ of the $i$-th
device in the $\slotIndex$-th slot $[ \slotIndex \Delta\tau, (\slotIndex+1) \Delta\tau ]$ is directly 
proportional to $\sx{i}_\slotIndex$, i.e., $g^i_\slotIndex (\sx{i}_\slotIndex ) = c^i \sx{i}_\slotIndex$.

Thus, optimization problem~\eqref{eq:minimax_starting_problem} for this scenario is
\begin{align}
\begin{split}
  \min_{\sx{1}, \ldots, \sx{N},P}  \: & \: P 
  \\[1.2ex]
  \subj \: & \:  A_i \sx{i} \preceq b_i, \; \sx{i}\in[0,1]^\slotUB, \hspace{0.2cm} i\in\until{N}\\[1.2ex]
           \: & \: \sum_{i=1}^N c^i \sx{i}_\slotIndex \leq P, \hspace{1.6cm} \slotIndex\in\until{\slotUB},
\end{split}
\label{eq:simulated_problem}
\end{align}
where $A_i$ and $b_i$ encode the constraints due to the discrete-time dynamics,
the temperature constraint $T^i_\slotIndex \in [ T_{min}, T_{max} ]$ and the
known forcing term $\delta^i_s$.  Notice that the local constraint set is
$X^i := \big\{ \sx{i} \in \real^\slotUB \mid A_i \sx{i} \preceq b_i \text{ and }
\sx{i}\in[0,1]^\slotUB \big\}$.

We choose each $\delta^i_{\slotIndex}$ to be constant for an interval of $5$
slots and zero otherwise. The nonzero values are set in the central part of the
entire simulation horizon $\until{S}$ by randomly shifting the center.
Then, we randomly choose the heterogeneous power consumption coefficient
$c^i\in \real$ of each device from a set of five values, drawn from a uniform
distribution in $[ 1,3 ]$.
Finally, we consider $N=20$ agents communicating
according to an undirected connected Erd\H{o}s-R\'enyi random graph $\GG$ with
parameter $0.2$. We consider a horizon of $\slotUB=60$. Finally, we
used a diminishing step-size sequence in the form $\gamma(t) = t^{-0.8}$,
satisfying Assumption~\ref{ass:step-size}.

In Figure~\ref{fig:rho} we show the evolution at each algorithm iteration $t$ of
the local objective functions $\srho{i}(t)$, $i\in\until{N}$, (solid lines)
which converge to stationary values.
Moreover, we also plot their sum $\sum_{i=1}^N\srho{i}(t)$ (dashed line) and the
value $P(t)$ (dotted line), introduced in~\eqref{eq:Pt_definition}. As proven in
Corollary~\ref{cor:eta_rho}, both of them asymptotically converge to the
centralized optimal cost $P^\star$ of problem~\eqref{eq:simulated_problem}. It
is worth noting that, at each iteration $t$, the curve $P(t)$ stays above the
optimal value $P^\star$ and below the curve $\sum_{i=1}^N\srho{i}(t)$, i.e.,
condition~\eqref{eq:carabinieri} is satisfied.
\begin{figure}[!htbp]
\centering
  \includegraphics[scale=1]{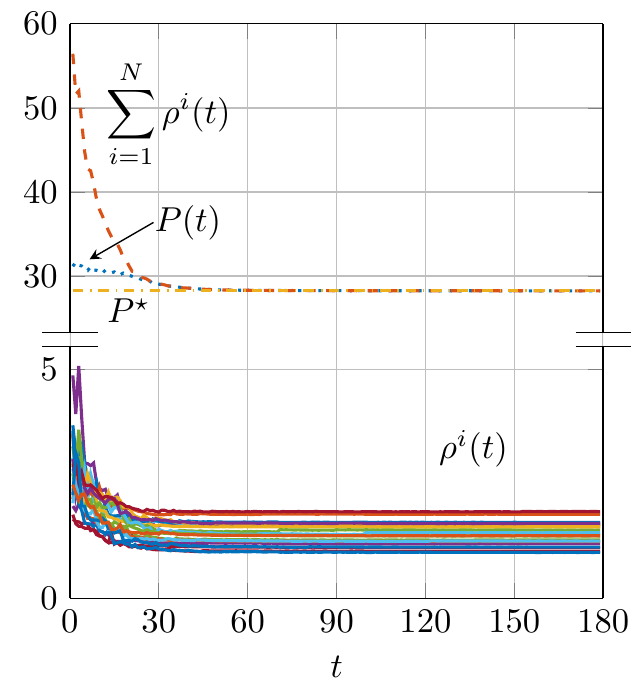}
  \caption{
    Evolution of $\srho{i}(t)$, $i\in\until{N}$, (solid lines), their sum
    $\sum_{i=1}^N\srho{i}(t)$ (dashed line), $P(t)$ (dotted line) and
    (centralized) optimal cost $P^\star$ (dash-dotted line).
    }
  \label{fig:rho}
\end{figure}

In Figure~\ref{fig:xx} the local solutions at the last algorithm iteration are
depicted. We denote them ${\sx{i}}^\star$, $i\in\until{N}$, to highlight that
they satisfy the cost optimality up to the required tolerance $10^{-3}$.
We also plot the resulting aggregate optimal consumption, i.e.,
$\sum_{i=1}^N c^i {\sx{i}}^\star$, which, as expected, in fact shaves off the
power demand peak.
\begin{figure}[!htbp]
\centering
  \includegraphics[scale=1]{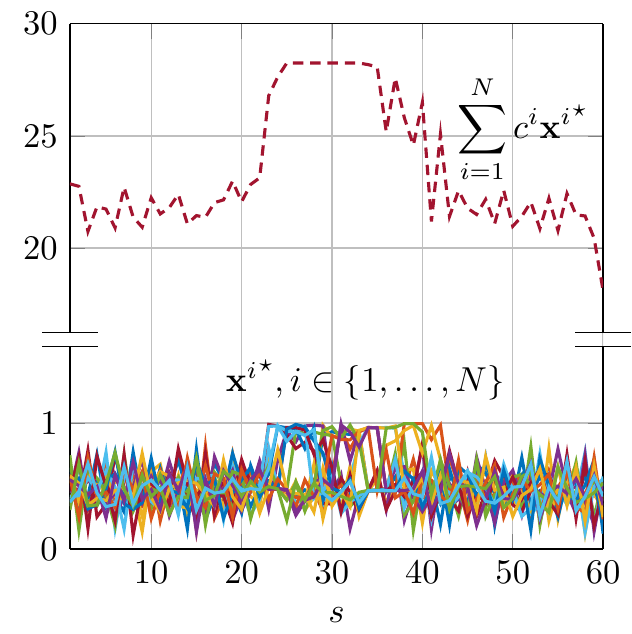}
  \caption{
    Profile of optimal solutions ${\sx{i}}^\star$ (solid lines), and $\sum_{i=1}^N c^i  {\sx{i}_\slotIndex}^\star$
    (dashed line) over the horizon $\until{\slotUB}$.
    }
  \label{fig:xx}
\end{figure}

Moreover, the optimal local solutions satisfy the box constraint $[0,1]$ for
each slot $\slotIndex \in\until{\slotUB}$. In fact, as we have proven, the
algorithm converges in an interior point fashion, i.e., the local constraint at
each node $i\in\until{N}$, is satisfied at all the algorithm iterations. As an
example, in Figure~\ref{fig:xx1_t} we depict the behavior of the components of
$\sx{1}(t)$.
\begin{figure}[!htbp]
\centering
  \includegraphics[scale=1]{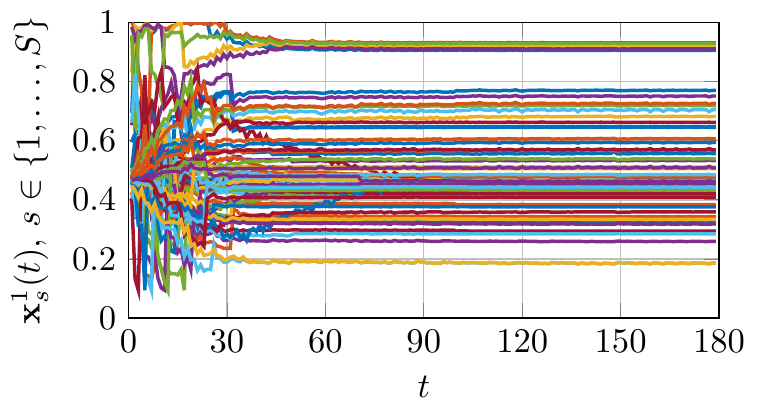}
  \caption{
    Evolution of $\sx{1}_\slotIndex(t)$, $\slotIndex\in \until{\slotUB}$.
    }
  \label{fig:xx1_t}
\end{figure}

In Figure~\ref{fig:primal_violations} (left) we show, the violation of the
coupling constraints, for all $\slotIndex\in\until{\slotUB}$ at each iteration $t$. As
expected, the violations asymptotically go to nonnegative values, consistently
with the asymptotic primal feasibility proven in the previous section.

In Figure~\ref{fig:primal_violations} (right) the difference
$\sum_{i=1}^N \left(g^i_s(\sx{i}_s(t)) - \srho{i}(t)\right)$ is also shown,
which is always nonnegative consistently with
equation~\eqref{eq:feasibility_condition}.
\begin{figure}[!htbp]
  \centering
  \includegraphics[scale=1]{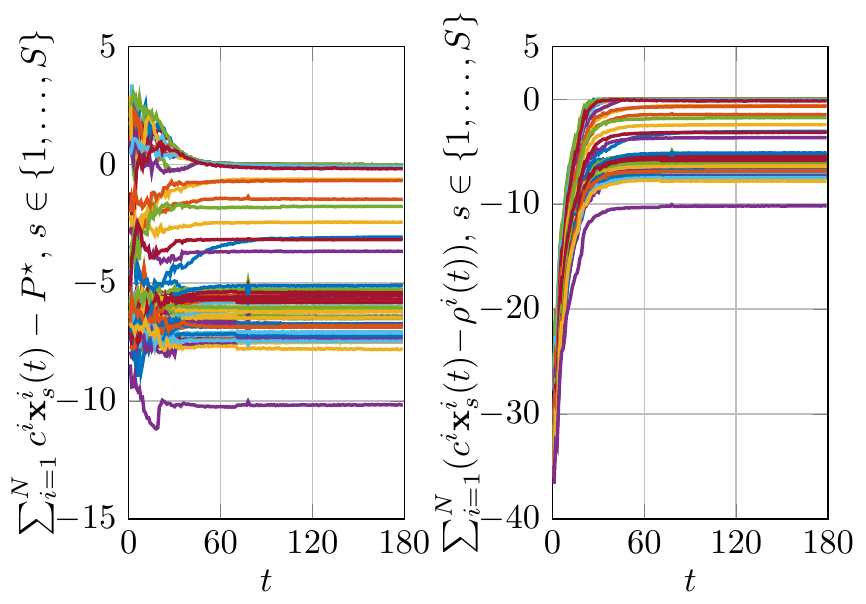}
  \caption{
    Evolution of primal violations of solutions ${\sx{i}}(t)$, $i\in\until{N}$.
  }
  \label{fig:primal_violations}
\end{figure}

Finally, in Figure~\ref{fig:cost} it is shown the convergence rate of the
distributed algorithm, i.e., the difference between the centralized optimal
cost  $P^\star$ and the sum of the local costs $\sum_{i=1}^N \srho{i} (t)$,
in logarithmic scale.  It can be seen that the proposed algorithm converges to
the optimal cost with a sublinear rate $O(1/\sqrt{t})$ as expected for a
subgradient method. Notice that the cost error is not monotone since the
subgradient algorithm is not a descent method. 
\begin{figure}[!htbp]
\centering
  \includegraphics[scale=1]{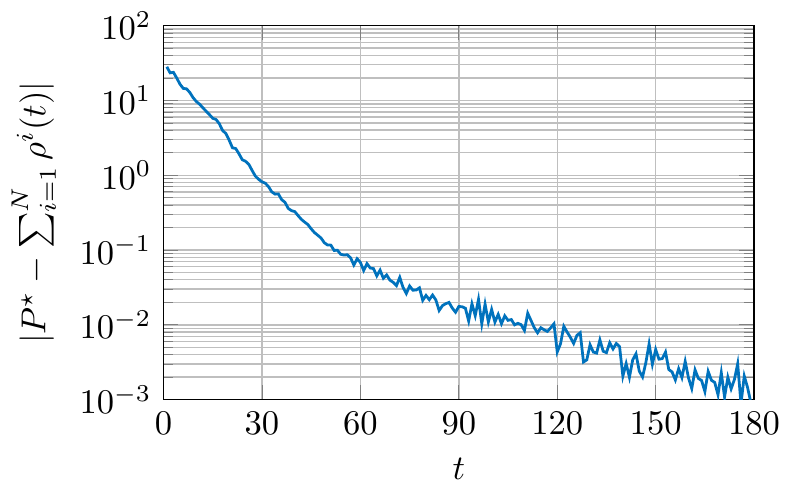}
  \caption{
    Evolution of the cost error, in logarithmic scale.
    }
  \label{fig:cost}
\end{figure}


\section{Conclusions}
\label{sec:conclusions}
In this paper we have introduced a novel distributed min-max optimization
framework motivated by peak minimization problems in Demand Side
Management. Standard distributed optimization algorithms cannot be applied to
this problem set-up due to a highly nontrivial coupling in the objective
function and in the constraints.
We proposed a distributed algorithm based on the combination of duality methods
and properties from min-max optimization.
Specifically, by means of duality theory, a series of equivalent problems are
set-up, which lead to a separable and sparse optimization problem. A subgradient
method applied to the resulting problem results into a distributed algorithm by
suitably applying properties from min-max optimization.  Despite the complex
derivation, the algorithm has a very simple structure at each node.
Theoretical results are corroborated by a numerical example on peak minimization
in Demand Side Management.


\appendix
\renewcommand{\theequation}{\thesubsection.\arabic{equation}}
\renewcommand{\thetheorem}{\thesubsection.\arabic{theorem}}

\subsection{Optimization and Duality}
\label{app:optimization_duality}
Consider a constrained optimization problem, addressed as primal problem,
having the form
\begin{align}
\begin{split}
  \min_{ z \in Z } \:& \: f(z)
  \\
  \subj \: & \: g(z) \preceq 0
\end{split}
\label{eq:appendix_primal}
\end{align}
where $Z \subseteq \real^N$ is a convex and compact set,
$\map{f}{\real^N}{\real}$ is a convex function and $\map{g}{\real^N}{\real^\slotUB}$
is such that each component $\map{g_\slotIndex}{\real^N}{\real}$,
$\slotIndex \in \until{\slotUB}$, is a convex function.

The following optimization problem
\begin{align}
\begin{split}
  \max_{\mu} \:& \: q(\mu)
  \\
  \subj \: & \: \mu \succeq 0
\end{split}
\label{eq:appendix_dual}
\end{align}
is called the dual of problem~\eqref{eq:appendix_primal}, where
$\map{q}{\real^\slotUB}{\real}$ is obtained by minimizing with respect to $z \in Z$
the Lagrangian function $\LL (z,\mu) := f(z) + \mu^\top g(z)$, i.e.,
$q(\mu) = \min_{z \in Z} \LL(z,\mu)$. Problem~\eqref{eq:appendix_dual} is
well posed since the domain of $q$ is convex and $q$ is concave on its domain.

It can be shown that the following inequality holds
\begin{align}
  \inf_{z \in Z } \sup_{\mu \succeq 0} \LL(z, \mu) \ge \sup_{\mu\succeq 0} \inf_{z \in X } \LL(z,\mu),
  \label{eq:appendix_weak_duality}
\end{align}
which is called weak duality.
When in~\eqref{eq:appendix_weak_duality} the equality holds, then we say
that strong duality holds and, thus, solving the primal
problem~\eqref{eq:appendix_primal} is equivalent to solving its dual
formulation~\eqref{eq:appendix_dual}. In this case the right-hand-side
problem in~\eqref{eq:appendix_weak_duality} is referred to as
\emph{saddle-point problem} of \eqref{eq:appendix_primal}.

\begin{definition}
  A pair $(z^\star , \mu^\star)$ is called a primal-dual optimal solution of
  problem~\eqref{eq:appendix_primal} if $z^\star\in Z$ and
  $\mu^\star\succeq 0$, and $(z^\star , \mu^\star)$ is a saddle point of the
  Lagrangian, i.e.,
  \begin{align*}
    \LL (z^\star,\mu ) \le \LL (z^\star,\mu^\star) \le \LL (z,\mu^\star)
  \end{align*}
  for all $z\in Z$ and $\mu\succeq 0$.\oprocend
\label{def:primal_dual_pair}
\end{definition}

A more general min-max property can be stated. Let $Z \subseteq \real^N$ and
$W \subseteq \real^\slotUB$ be nonempty convex sets.  Let
$\phi : Z \times W \to \real$, then the following inequality
\begin{align*}
  \inf_{z\in Z} \sup_{w \in W}  \phi (z,w) \ge \sup_{w \in W} \inf_{z\in Z} \phi (z,w)
\end{align*}
holds true and is called the \emph{max-min} inequality. When the equality holds, then
we say that $\phi$, $Z$ and $W$ satisfy the \emph{strong max-min} property or the
\emph{saddle-point} property.

The following theorem gives a sufficient condition for the strong max-min property
to hold.

\begin{proposition}[{\cite[Propositions~4.3]{bertsekas2009min}}]
  Let $\phi$ be such that (i) $\map{\phi (\cdot,w)}{Z}{\real}$ is convex and closed
  for each $w \in W$, and (ii) $\map{-\phi(z,\cdot)}{W}{\real}$ is convex and closed for
  each $z \in Z$.
  Assume further that $W$ and $Z$ are convex compact sets. Then
  \begin{align*}
    \sup_{w \in W} \inf_{z \in Z} \phi (z,w) = \inf_{z \in Z} \sup_{w \in W}  \phi (z,w)
  \end{align*}
  and the set of saddle points is nonempty and compact.~\oprocend
  \label{prop:saddle_point}
\end{proposition}


\subsection{Subgradient Method}
\label{app:subgradient_method}

Consider the following (constrained) optimization problem
\begin{align}
  \min_{z\in Z} f (z)
  \label{eq:subgradient_problem}
\end{align}
with $\map{f}{\real^N}{\real}$ a convex function and $Z \subseteq \real^N$
a closed, convex set.

A vector $\widetilde{\nabla} f( z ) \in\real^N$ is called a subgradient of the convex
function $f$ at $z\in\real^N$ if
$f(y) \ge f(z) + \widetilde{\nabla} f( z )  (y - z)$ for all $y\in\real^N$.
The (projected) subgradient method is the iterative algorithm
given by
\begin{align}
  z(t+1) = P_{Z} \Big( z(t) - \gamma(t) \widetilde{\nabla} f( z(t) ) \Big)
\label{eq:subgradient_iteration}
\end{align}
where $t\ge 0$ denotes the iteration index, $\gamma(t)$ is the
step-size, $\widetilde{\nabla} f( z(t) )$ denotes a subgradient of $f$ at
$z(t)$, and $P_Z(\cdot)$ is the Euclidean projection onto $Z$.

The following standard assumption is usually needed to guarantee convergence of
the subgradient method.

\begin{assumption}
\label{ass:step-size}
  The sequence $\{ \gamma(t)\}$, with $\gamma(t) \ge 0$ for all $t\ge 0$,
  satisfies the diminishing condition
  \begin{align}\notag
    \lim_{t\to \infty} \gamma(t) = 0, \:\:
    \sum_{t=1}^{\infty} \gamma(t) = \infty, \:\:
    \sum_{t=1}^{\infty} \gamma(t)^2 < \infty. \eqoprocend
  \end{align}
\end{assumption}

The following proposition formally states the convergence of the subgradient method.
\begin{proposition}[{\cite[Proposition 3.2.6]{bertsekas2015convex}}]
  Assume that the subgradients $\widetilde{\nabla} f(z)$ are bounded for all
  $z \in Z$ and the set of optimal solutions is nonempty.  Let the step-size
  $\{\gamma(t)\}$ satisfy the diminishing condition in
  Assumption~\ref{ass:step-size}. Then the subgradient method
  in~\eqref{eq:subgradient_iteration} applied to
  problem~\eqref{eq:subgradient_problem} converges in objective value and
  sequence $\{ z(t)\}$ converges to an optimal solution.~\oprocend
\label{prop:subgradient_convergence}
\end{proposition}

  \bibliographystyle{IEEEtran}
  \bibliography{distributed_min-max}

\begin{IEEEbiography}
  [{\includegraphics[scale=0.16]{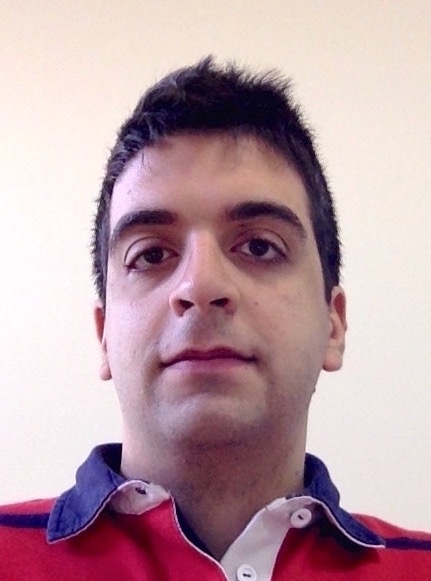}}]
  {Ivano Notarnicola}
  has been a Ph.D. student in Engineering of Complex Systems at the Universit\`a 
  del Salento (Lecce, Italy) since November 2014. He received the Laurea degree 
  ``summa cum laude'' in Computer Engineering from the Universit\`a del Salento 
  in 2014. He was a visiting scholar at the University of Stuttgart from March to June 2014. 
  His research interests include distributed optimization and randomized algorithms.
\end{IEEEbiography}

\begin{IEEEbiography}
[{\includegraphics[width=1in,height=1.25in,clip,keepaspectratio]{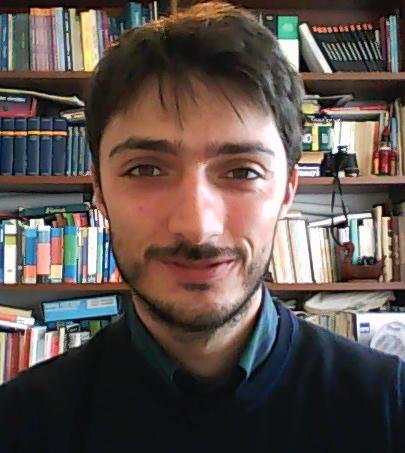}}]
{Mauro Franceschelli}
(M'11) was born in 1983. He received the B.S. and M.S. degree in Electronic Engineering ``cum laude'' in 2005 
and 2007 from the University of Cagliari, Italy. He received the PhD from the same University in 2011 where he 
continued its research collaboration with a postdoc until 2015. He has spent visiting periods at the Royal Institute 
of Technology (KTH), the Georgia Institute of Technology (GaTech), the University of California at Santa Barbara 
(UCSB) and Xidian University, Xi'an, China. Since 2015 he is Assistant Professor (RTD) at the Department of Electrical 
and Electronic Engineering, University of Cagliari, Italy, as recipient of a grant founded by the Italian Ministry of 
Education, University and Research (MIUR) under the 2014 call  ``Scientific Independence of Young Researchers'' (SIR). 
He has published more than 50 papers in top international journals and conference proceedings. His research interests 
include consensus problems, gossip algorithms, multi-agent systems, multi-robot systems, distributed optimization 
and electric demand side management.
\end{IEEEbiography}

\begin{IEEEbiography}
  [{\includegraphics[scale=0.075]{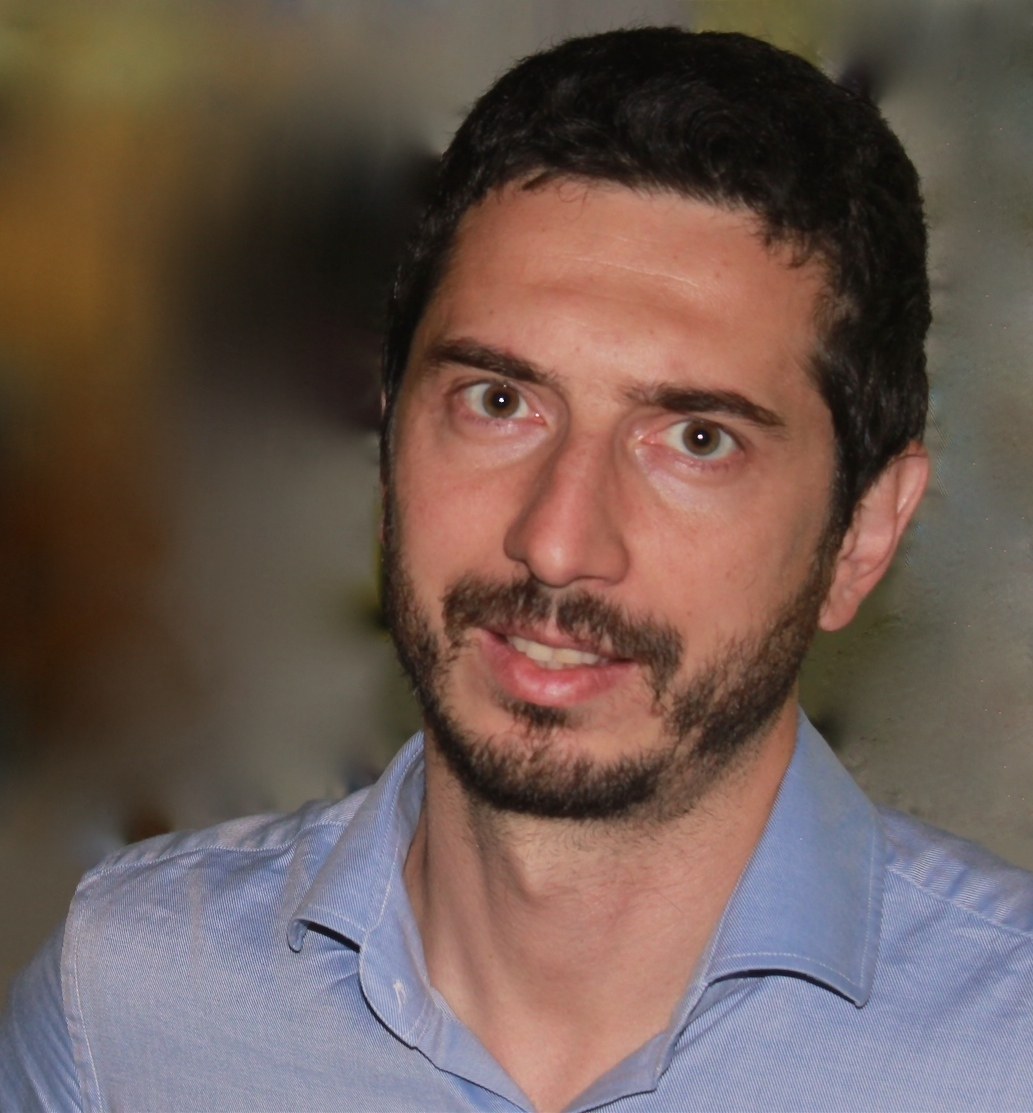}}]
  {Giuseppe Notarstefano}
  is Associate Professor at the Universit\`a del Salento (Lecce, Italy), where he was Assistant Professor 
  (Ricercatore) from February 2007 to May 2016. He received the Laurea degree ``summa cum laude'' 
  in Electronics Engineering from the Universit\`a di Pisa in 2003 and the Ph.D. degree in Automation 
  and Operation Research from the Universit\`a di Padova in April 2007. He has been visiting scholar 
  at the University of Stuttgart, University of California Santa Barbara and University of Colorado Boulder. 
  His research interests include distributed optimization, cooperative control in complex networks, 
  applied nonlinear optimal control, and trajectory optimization and maneuvering of aerial and car vehicles.
  He serves as an Associate Editor in the Conference Editorial Board of the IEEE Control Systems Society 
  and for other IEEE and IFAC conferences. He coordinated the VI-RTUS team winning the International Student 
  Competition Virtual Formula 2012. He is recipient of an ERC Starting Grant 2014.
\end{IEEEbiography}

\end{document}